\documentclass[12pt,twoside]{article}
\usepackage{amsfonts}
\usepackage{amsmath}
{\markboth{\sl A. Jemai}{\sl
 }

\newtheorem{lemma}{Lemma}[section]

\newtheorem{theorem}{Theorem }[section]
\newtheorem{remark}{Remark}[section]

\newtheorem{proof}{Proof.}[section]
\numberwithin  {equation}{section}

\begin{document}

\title{Dunkl-spherical maximal function}

\author{ Abdessattar Jemai}

\maketitle

\begin{abstract} In this paper, we study the $L^p$-boundedness of the spherical maximal
function associated to the Dunkl operators.
 \end{abstract}
\section{Introduction and backgrounds}
In \cite{S2}, E. Stein introduced the spherical maximal function by
$$M(f)(x)=\sup_{r>0}\int_{\mathbb{S}^{d-1}}f(x-ry)d\sigma(y),$$
where $d\sigma$ is the surface measure on $\mathbb{S}^{d-1}$ and
showed that $M$ is bounded from $L^{p}(\mathbb{R}^d)$ to
$L^{p}(\mathbb{R}^d)$ for the optimal range $p> \frac{d}{d-1}$ with
$d\geq3$. The case $d = 2$ was proved by Bourgain in \cite{B}. The
aim of this work is to extend these results to the Dunkl
setting.\par

To begin, we recall some results in Dunkl theory (see \cite{D, J,
R1, R2, R4, TX}) and we refer for more details to the
survey \cite{R3}.\\
Let $G\subset O(\mathbb{R}^{d})$ be a finite reflection group associated to a reduced
 root system $R$. For $\alpha\in R$, we denote by
$\mathbb{H}_\alpha$ the hyperplane orthogonal to $\alpha$. We denote
by $k$ a nonnegative multiplicity function defined on $R$ with the
property that $k$ is $G$-invariant. For a given
$\beta\in\mathbb{R}^d\backslash\bigcup_{\alpha\in R}
\mathbb{H}_\alpha$, we fix a positive subsystem $R_+=\{\alpha\in R:
\langle \alpha,\beta\rangle>0\}$. We associate with $k$ the index
$\displaystyle\gamma = \sum_{\xi \in R_+} k(\xi)$ and a weighted
measure $\nu_k$ given by
\begin{eqnarray*}d\nu_k(x):=w_k(x)dx\quad
\mbox{ where }\;\;w_k(x) = \prod_{\xi\in R_+} |\langle
\xi,x\rangle|^{2k(\xi)}, \quad x \in
\mathbb{R}^d.\end{eqnarray*}Further, we introduce the Mehta-type
constant $c_k$ by
$$c_k = \left(\int_{\mathbb{R}^d} e^{- \frac{\|x\|^2}{2}}
w_k (x)dx\right)^{-1}.$$ For every $1 \leq p \leq + \infty$, we
denote by $L^p_k(\mathbb{R}^d)$,
 the spaces $L^{p}(\mathbb{R}^d,
d\nu_k(x)),$
 and we use $\| \;\|_{p,k}$ as a shorthand for $\|\ \;\|_{L^p_k( \mathbb{R}^d)}$. \\By
using the homogeneity of degree $2\gamma$ of $w_k$, for a radial
function $f$ in $L^1_k ( \mathbb{R}^d)$, there exists a function $F$
on $[0, + \infty)$ such that $f(x) = F(\|x\|)$, for all $x \in
\mathbb{R}^d$. The function $F$ is integrable with respect to the
measure $r^{2\gamma+d-1}dr$ on $[0, + \infty)$ and we have
 \begin{eqnarray*}\int_{\mathbb{R}^d}  f(x)\,d\nu_k(x)&=&\int^{+\infty}_0
\Big( \int_{\mathbb{S}^{d-1}}f(ry)w_k(ry)d\sigma(y)\Big)r^{d-1}dr\\
&=&
 \int^{+\infty}_0
\Big( \int_{\mathbb{S}^{d-1}}w_k(ry)d\sigma(y)\Big)
F(r)r^{d-1}dr\\&= & d_k\int^{+ \infty}_0 F(r)
r^{2\gamma+d-1}dr,
\end{eqnarray*}
  where $S^{d-1}$
is the unit sphere on $\mathbb{R}^d$ with the normalized surface
measure $d\sigma$  and \begin{eqnarray*}d_k=\int_{\mathbb{S}^{d-1}}w_k
(x)d\sigma(x) = \frac{c^{-1}_k}{2^{\gamma +\frac{d}{2} -1}
\Gamma(\gamma + \frac{d}{2})}\;.  \end{eqnarray*}

The Dunkl operators $T_j,\; 1\leq j\leq d\,$ are the following
$k$-deformations of directional derivatives $\frac{\partial}{x_j} $
given by:
$$T_jf(x)=\frac{\partial f}{\partial x_j}(x)+\sum_{\alpha\in R_+}k(\alpha)
\alpha_j\,\frac{f(x)-f(\rho_\alpha(x))}{\langle\alpha,x\rangle}\,,\quad
f\in\mathcal{E}(\mathbb{R}^d)\,,\quad x\in\mathbb{R}^d\,,$$ where
$\rho_\alpha$ is the reflection on the hyperplane
$\mathbb{H}_\alpha$ and $\alpha_j=\langle\alpha,e_j\rangle,$
$(e_1,\ldots,e_d)$ being the canonical basis of $\mathbb{R}^d$.

Notice that in the case $k\equiv0$, the weighted function
$w_k\equiv1$,  the measure $\nu_k$ coincide with the Lebesgue
measure and the operator $T_{\xi}$ reduced to the corresponding
partial derivatives $\frac{\partial}{x_j} $. Therefore Dunkl
analysis can be viewed as a generalization of classical Fourier
analysis.

For $y \in \mathbb{C}^d$, the system
$$\left\{\begin{array}{lll}T_ju(x,y)&=&y_j\,u(x,y),\qquad1\leq j\leq d\,,\\  &&\\
u(0,y)&=&1\,.\end{array}\right.$$ admits a unique analytic solution
on $\mathbb{R}^d$, denoted by $E_k(x,y)$ and called the Dunkl
kernel. This kernel has a unique holomorphic extension to
$\mathbb{C}^d \times \mathbb{C}^d $. We have for all $\lambda\in
\mathbb{C}$ and $z, z'\in \mathbb{C}^d,\;
 E_k(z,z') = E_k(z',z)$,  $E_k(\lambda z,z') = E_k(z,\lambda z')$ and for $x, y
\in \mathbb{R}^d,\;|E_k(x,iy)| \leq 1$.\\

The Dunkl transform $\mathcal{F}_k$ is defined for $f \in
\mathcal{D}( \mathbb{R}^d)$ by
$$\mathcal{F}_k(f)(x) =c_k\int_{\mathbb{R}^d}f(y) E_k(-ix, y)d\nu_k(y),\quad
x \in \mathbb{R}^d.$$  We list some known properties of this
transform:
\begin{itemize}
\item[i)] The Dunkl transform of a function $f
\in L^1_k( \mathbb{R}^d)$ has the following basic property
\begin{eqnarray*}\| \mathcal{F}_k(f)\|_{\infty,k} \leq
 \|f\|_{ 1,k}\;. \end{eqnarray*}
\item[ii)] The Dunkl transform is an automorphism on the Schwartz space $\mathcal{S}(\mathbb{R}^d)$.
\item[iii)] When both $f$ and $\mathcal{F}_k(f)$ are in $L^1_k( \mathbb{R}^d)$,
 we have the inversion formula \begin{eqnarray*} f(x) =   \int_{\mathbb{R}^d}\mathcal{F}_k(f)(y) E_k( ix, y)d\nu_k(y),\quad
x \in \mathbb{R}^d.\end{eqnarray*}
\item[iv)] (Plancherel's theorem) The Dunkl transform on $\mathcal{S}(\mathbb{R}^d)$
 extends uniquely to an isometric automorphism on
$L^2_k(\mathbb{R}^d)$.
\end{itemize}

The Dunkl translation operators $\tau_x$, $x\in\mathbb{R}^d$ is defined on
$L^2_k(\mathbb{R}^d)$ by
\begin{eqnarray*}\mathcal{F}_k(\tau_x(f))(y)=E_k(i x, y)\mathcal{F}_k(f)(y)
\quad y\in\mathbb{R}^d.\end{eqnarray*}

As an operator on $L_k^2(\mathbb{R}^d)$, $\tau_x$ is bounded.
According to (\cite{TX}, Theorem 3.7), the operator $\tau_x$ can be
extended to the space of radial functions
$L^p_k(\mathbb{R}^d)^{rad},$ $1 \leq p \leq 2$ and we have for a
function $f$ in $L^p_k(\mathbb{R}^d)^{^{rad}}$,
\begin{eqnarray*}\|\tau_x(f)\|_{p,k} \leq \|f\|_{p,k}.\end{eqnarray*}
It was shown in \cite{R5} that if $f$ is a radial function in
$\mathcal{S}(\mathbb{R}^d)$ with $f(y)=\widetilde{f}(\|y\|)$, then
\begin{eqnarray}\label{tra}
  \tau_x (f)(y)=\int_{\mathbb{R}^{d}}\widetilde{f}(A(x,y,\eta))d\mu_{x}(\eta)
\end{eqnarray}
where $A(x,y,\eta)=\sqrt{\|x\|^{2}+\|y\|^{2}-2<y,\eta>}$ and $\mu_{x}$ is a probability measure supported in the convex hull
$co(G.x)$ of the $G$-orbit of $x$ in $\mathbb{R}^d$. We observe
that,
\begin{eqnarray}
\label{min} \eta \in co(G.x)\Longrightarrow \min_{g\in G}\|g.x-y\|\leq  A(x,y,\eta) \leq \max_{g\in
  G}\|g.x-y\|.
\end{eqnarray}
We collect below some useful facts:

\begin{itemize}
\item[i)] For all $x,y\in\mathbb{R}^d$, $\tau_x(f)(y)=\tau_y(f)(x)$.
\item[ii)] For  $f\in L_k^2(\mathbb{R}^d)\cap L_k^1(\mathbb{R}^d)$
\begin{eqnarray}\label{3frad}\int_{\mathbb{R}^d}\tau_x(f)(y)d\nu_k(y)=\int_{\mathbb{R}^d}f(y)d\nu_k(y).\end{eqnarray}
\item[iii)] For all $x\in\mathbb{R}^d$ and $f, g\in L_k^2(\mathbb{R}^d)$
\begin{eqnarray}\label{3prmi}\int_{\mathbb{R}^d}\tau_x(f)(y)g(y)d\nu_k(y)=\int_{\mathbb{R}^d}f(y)\tau_x(g)(y)d\nu_k(y).\end{eqnarray}
\end{itemize}
The Dunkl convolution product $\ast_k$ of two functions $f$ and $g$
in $L^2_k(\mathbb{R}^d)$ is given by
\begin{eqnarray*}(f\; \ast_k g)(x) = \int_{\mathbb{R}^d} \tau_x (f)(-y) g(y) d\nu_k(y),\quad
x \in \mathbb{R}^d.\end{eqnarray*} The Dunkl convolution product is
commutative and for $f,\,g \in \mathcal{D}( \mathbb{R}^d)$, we have
\begin{eqnarray}\label{ast}\mathcal{F}_k(f\,\ast_k\, g) =
\mathcal{F}_k(f) \mathcal{F}_k(g).\end{eqnarray} It was proved in
(\cite{TX}, Theorem 4.1) that when $g$ is a bounded radial function
in $L^1_k( \mathbb{R}^d)$, then the application $f\longrightarrow
f\,\ast_k\, g$ initially defined on the intersection of
$L^1_k(\mathbb{R}^d)$ and $L^2_k(\mathbb{R}^d)$ extends to
$L^p_k(\mathbb{R}^d)$, $1\leq p\leq +\infty$ as a bounded operator.
In particular, \begin{eqnarray*}\|f \ast_k g\|_{p,k} \leq
\|f\|_{p,k} \|g\|_{1,k}.\end{eqnarray*}

(see \cite{R5,R6} ) The Dunkl transform of $\sigma$ is given by
\begin{eqnarray}\label{sig}\mathcal{F}_k(\sigma)(x)=\frac{1}{c_k}\int_{\mathbb{S}^{d-1}} E_k(-ix, y)\omega_k(y)d\sigma(y)=
c_\gamma j_{\frac{2\gamma+d}{2}-1}( \|x\|),\end{eqnarray}
where $j_{\frac{2\gamma+d}{2}-1}$ is the Bessel function of the first type and $c_\gamma=\frac{1}{2^{\gamma+\frac{d}{2}-1}\Gamma(\gamma+\frac{d}{2})}$.\\
In particular (see \cite{S3}), the function
\begin{eqnarray}\label{sigm}\mathcal{F}_k(\sigma)(r\xi)=O(r^{-\frac{2\gamma+d-1}{2}}),\qquad r\rightarrow +\infty.\end{eqnarray}
\begin{eqnarray}\label{par}\frac{1}{r}\frac{\partial}{\partial r}\mathcal{F}_k(\sigma)(r\xi)=O(r^{-\frac{2\gamma+d+1}{2}}),\qquad r\rightarrow +\infty.\end{eqnarray}

Along this paper we use $C$ to denote a suitable positive constant
which is not necessarily the same in each occurrence and we write
for $x \in \mathbb{R}^d, \|x\| = \sqrt{\langle x,x\rangle}$.
Furthermore, we denote by

$\bullet\quad \mathcal{E}(\mathbb{R}^d)$ the space of infinitely
differentiable functions on $\mathbb{R}^d$.

$\bullet\quad \mathcal{S}(\mathbb{R}^d)$ the Schwartz space of
functions in $\mathcal{E}( \mathbb{R}^d)$ which are rapidly
decreasing as well as their derivatives.

$\bullet\quad \mathcal{D}(\mathbb{R}^d)$ the subspace of
$\mathcal{E}(\mathbb{R}^d)$ of compactly supported functions.

\section{Dunkl-spherical maximal function}
For $f\in L^2_k(\mathbb{R}^d)$, we define the maximal function $\mathcal{M}_{k}f$  by
\begin{eqnarray}\label{max}\mathcal{M}_{k}(f)(x)=\sup_{r>0}\frac{1}{\nu_k(B(0,r))}\left|\int_{B(0,r)}\tau_{x}f(y)d\nu_{k}(y)\right|,\end{eqnarray}
where $B(0,r)$ is the ball of radius $r$ centered at $0$ and $x\in\mathbb{R}^d$.\\
It was proved in \cite{TX} that the maximal function  is bounded on $f\in L^p_k(\mathbb{R}^d)$
 for $1<p\leq\infty$,
and of weak type $(1,1)$, for $f\in L^1_k(\mathbb{R}^d)$
that is $a>0$, \begin{eqnarray}\label{fai}\int_{E(a)}d\nu_{k}(x)\leq \frac{C}{a}\|f\|_{1,k}\end{eqnarray}
where $E(a)=\{x: \mathcal{M}_{k}f(x)>a\}$ and $C$ is a constant independent of $a$
and $f$.
\par As
in \cite{mejj}, we define the spherical mean operator on \\$A_k(\mathbb{R}^d)=\{ f\in L^1_k(\mathbb{R}^d):\; \mathcal{F}_{k}(f)\in L^1_k(\mathbb{R}^d)\}$,
for $x\in\mathbb{R}^d$ and $r>0$ by
\begin{eqnarray*}\mathbf{S}_{r}(f)(x)&=&\frac{1}{d_k}\int_{\mathbb{S}^{d-1}}\tau_{x}(f(-ry))\omega_k(y)d\sigma(y)\\&=&
\frac{1}{d_k}\int_{r\mathbb{S}^{d-1}}\tau_{x}(f(-y))\omega_k(\frac{y}{r})d\sigma_{r}(y).
\end{eqnarray*}
Now, the Dunkl-spherical maximal function $M(f)$ is given by
$$M(f)(x)=\sup_{r>0}|\mathbf{S}_{r}(f)(x)|,\quad x\in\mathbb{R}^d.$$
\begin{theorem}
Let $2\gamma+d\geq2$ and $\frac{2\gamma+d}{2\gamma+d-1}<p<2\gamma+d$. Then there exists
a constant $C>0$ such that for all $f\in L^p_k(\mathbb{R}^d)$,
\begin{eqnarray}\label{for} \|M(f)\|_{p,k}\leq C\, \|f\|_{p,k}.\end{eqnarray}
\end{theorem}

Before proving the theorem, we need to establish some useful
results. In fact, fix a function $\psi_{0}\in
\mathcal{S}(\mathbb{R}^d)$ which is the Dunkl transform of a
$\mathcal{C}^{\infty}$-radial function with compact support such
that:
  \begin{eqnarray}\label{psi}
 \psi_{0}(0)=1,\qquad\big(\frac{\partial^{i}}{\partial r^{i}}\psi_{0}\big)(0)=0,
 \end{eqnarray}
 for $1\leq i< \frac{2\gamma +d}{2}$ and where $\frac{\partial}{\partial r}$ denotes the derivation in the radial direction.
To obtain a such function, taking  $\psi\in \mathcal{S}(\mathbb{R}^d)$  with  $\mathcal{F}_k(\psi)$ is a $\mathcal{C}^{\infty}$-
radial function with  compact support and
$$\psi(0)\neq 0 \quad\text{and}\quad\psi_{0}(r\xi)=\big(\sum_{j=0}^{[\frac{2\gamma +d}{2}]}a_{j}r^{j}\big)\psi(r\xi),$$
for $\xi\in \mathbb{S}^{d-1}$.
The coefficients $a_{j}$ are solutions of triangular system given by the conditions (\ref{psi})
and $[\frac{2\gamma +d}{2}]$ is the least integer not less than $\frac{2\gamma +d}{2}$.
\par Now set the functions $\psi_j$ with
\begin{eqnarray}\label{ps2}
\psi_{1}(y)=\psi_{0}(\frac{y}{2})-\psi_{0}(y) \quad \text{and} \quad \psi_{j}(y)=\psi_{1}(2^{-(j-1)}y),\quad j\geq 1.
\end{eqnarray}
Thus, $\psi_{j}$ is radial and $\mathcal{F}_k(\psi_j)$ is a compact supported functions. It follows that  for some constants  $t,C>0$ we have
\begin{eqnarray}\label{ps1}
|\psi_{1}(y)|\leq C\|y\|^{\frac{2\gamma +d}{2}},\quad\text{if}\quad \|y\|\leq t,
\end{eqnarray} and
\begin{eqnarray}\label{ps}\sum_{j=0}^{+\infty}\psi_{j}(y)=1, \qquad  y\in\mathbb{R}^{d}.\end{eqnarray}
Let $m_{j}=\mathcal{F}_{k}(\sigma)\psi_{j}$ and let $\varphi_{j}$ the function in $\mathcal{S}(\mathbb{R}^d)$ such that
$\mathcal{F}_{k}(\varphi_{j})=m_{j}$.\\
It follows that for $f\in\mathcal{S}(\mathbb{R}^d)$,
 $$\mathbf{S}_{r}(f)(x)=\sum_{j=0}^{+\infty}f\ast_{k}\varphi_{j}.$$
In fact, this can be done because from (\ref{ps}), we have
\begin{eqnarray*}
\mathcal{F}_{k}(\sigma)(y)=\sum_{j=0}^{+\infty}\mathcal{F}_{k}(\sigma)(y)\psi_{j}(y)
=\sum_{j=0}^{+\infty}\mathcal{F}_{k}(\varphi_{j})(y)
\end{eqnarray*}
and we can write
\begin{eqnarray*}
\mathcal{F}_{k}(\mathbf{S}_{r}(f))(y)=\sum_{j=0}^{+\infty}\mathcal{F}_{k}(f\ast_{k}\varphi_{j})(y),
\end{eqnarray*}
which implies with the inversion formula that
 \begin{eqnarray*}
\mathbf{S}_{r}(f)(y)=\mathcal{F}^{-1}_{k}\Big(\sum_{j=0}^{+\infty}\mathcal{F}_{k}(f\ast_{k}\varphi_{j})\Big)(y)=
\int_{\mathbb{R}^d}\sum_{j=0}^{+\infty}\mathcal{F}_{k}(f\ast_{k}\varphi_{j})(z) E_k( iy, z)d\nu_k(z).
\end{eqnarray*}
To  interchange the sum and the integral, we proceed as follows: \\
From (\ref{ast}), we have\\
$\displaystyle{\sum_{j=1}^{+\infty}\Big(\int_{\mathbb{R}^d}\Big|\mathcal{F}_{k}(f\ast_{k}\varphi_{j})(z) E_k( iy, z)\Big|d\nu_k(z)\Big)}$
\begin{eqnarray*}
&\leq&\sum_{j=1}^{+\infty}\Big(\int_{\mathbb{R}^d}|\mathcal{F}_{k}(\varphi_{j})(z)||\mathcal{F}_{k}(f)(z)|d\nu_k(z)\Big)\\
&\leq&\sum_{j=1}^{+\infty}\Big(\int_{\mathbb{R}^d}|\psi_{j}(z)||\mathcal{F}_{k}(f)(z)|d\nu_k(z)\Big)\\&\leq&
\sum_{j=1}^{+\infty}\Big(\int_{\|z\|\leq 2^{j-1}t}|\psi_{j}(z)||\mathcal{F}_{k}(f)(z)|d\nu_k(z)+
\int_{\|z\|\geq 2^{j-1}t}|\psi_{j}(z)||\mathcal{F}_{k}(f)(z)|d\nu_k(z)\Big).
\end{eqnarray*}
By using (\ref{ps2}) and (\ref{ps1}), we obtain
\begin{eqnarray*}
\int_{\|y\|\leq 2^{j-1}t}|\psi_{j}(z)||\mathcal{F}_{k}(f)(z)|d\nu_k(z)&\leq&
 c2^{-(j-1)\frac{2\gamma +d}{2}}\int_{\|y\|\leq 2^{j-1}t}\|z\|^{\frac{2\gamma +d}{2}}|\mathcal{F}_{k}(f)(z)|d\nu_k(z)\\&\leq&
 c2^{-(j-1)\frac{2\gamma +d}{2}}\int_{\mathbb{R}^d}\|z\|^{\frac{2\gamma +d}{2}}|\mathcal{F}_{k}(f)(z)|d\nu_k(z),
\end{eqnarray*}
and
\begin{eqnarray*}
\int_{\|z\|\geq 2^{j-1}t}|\psi_{j}(z)||\mathcal{F}_{k}(f)(z)|d\nu_k(z)&=&\int_{\frac{\|z\|}{2^{j-1}t}\geq 1}|\psi_{j}(z)||\mathcal{F}_{k}(f)(z)|d\nu_k(z)\\
&\leq&2^{-(j-1)}t\int_{\mathbb{R}^d}\|z\||\psi_{j}(z)||\mathcal{F}_{k}(f)(z)|d\nu_k(z)\\
&\leq&2^{-(j-1)}t\int_{\mathbb{R}^d}\|z\||\mathcal{F}_{k}(f)(z)|d\nu_k(z),
\end{eqnarray*}
therefore  $\displaystyle{\sum_{j=1}^{+\infty}\Big(\int_{\mathbb{R}^d}\Big|\mathcal{F}_{k}(f\ast_{k}\varphi_{j})(z) E_k( iy, z)\Big|d\nu_k(z)\Big)}$
converge.\\
For all $j\geq0$ and $r>0$, we define  the function
$\varphi_{j,r}(x)=r^{-2\gamma -d}\varphi_{j}(\frac{x}{r})$. Then we
can write,
$$\mathbf{S}_{r}(f)(x)=\sum_{j=0}^{+\infty}f\ast_{k}\varphi_{j,r}.$$
Hence for  $f\in\mathcal{S}(\mathbb{R}^d)$, we obtain
\begin{eqnarray}\label{princ}M(f)\leq\sum_{j=0}^{+\infty}M_{\varphi_{j}}(f),\end{eqnarray}
where $M_{\varphi_{j}}(f)(x)=\displaystyle\sup_{r>0}|f\ast_{k} \varphi_{j,r}(x)|
=\sup_{r>0}\left|\int_{\mathbb{R}^d}\tau_{x}(f)(y) \varphi_{j,r}(y)d\nu_{k}(y)\right|.$\\
To prove the theorem, it suffices to establish an inequality of the form
$$\|M_{\varphi_{j}}(f)\|_{p,k}\leq C_{j,p}\|f\|_{p,k},$$
with $\displaystyle\sum_{j=0}^{+\infty} C_{j,p}<+\infty$.
\begin{lemma}
There exists a constant $C>0$ such that, for any $x\in\mathbb{R}^d$ and $j\geq0$,
\begin{eqnarray}\label{lem}|\varphi_{j}(x)|\leq C \frac{2^{j}}{(1+\|x\|)^{2\gamma+d+1}}.\end{eqnarray}
\end{lemma}
\begin{proof}
Remember that $\varphi_{j}=\mathbf{S}_{r}(\mathcal{F}^{-1}_{k}(\psi_{j}))$ for $j\geq1$. We have,
$$\mathcal{F}^{-1}_{k}(\psi_{j})(x)=2^{(j-1)(2\gamma+d)}\mathcal{F}^{-1}_{k}(\psi_{1})(2^{j-1}x).$$
Since $\mathcal{F}^{-1}_{k}(\psi_{0})$ and $\mathcal{F}^{-1}_{k}(\psi_{1})$ are in $\mathcal{S}(\mathbb{R}^d)$, then
\begin{eqnarray}\label{varj}|\mathcal{F}^{-1}_{k}(\psi_{0})(x)|\; \mbox{and} \;|\mathcal{F}^{-1}_{k}(\psi_{1})(x)|
\;\;\mbox{are bounded
by}\;\;\frac{C}{(1+\|x\|)^{2\gamma+d+1}}.\end{eqnarray} Taking
$\phi_{j}(x,y,\xi)=\mathcal{F}^{-1}_{k}(\psi_{1})(2^{j-1}\sqrt{\|x\|^2+\|y\|^2-2<y,\xi>})$
for $j\geq1$ and using (\ref{tra}), we get
\begin{eqnarray}\label{tau}
|\varphi_{j}(x)|&=&\Big|\int_{\mathbb{S}^{d-1}}\tau_{x}(\mathcal{F}^{-1}_{k}(\psi_{j}))(-y) \omega_k(y)d\sigma(y)\Big|\nonumber
\\&\leq&2^{(j-1)(2\gamma+d)}\int_{\mathbb{S}^{d-1}}\Big(\int_{\mathbb{R}^d}\Big|\phi_{j}(x,y,\xi)\Big|
d\mu_{x}(\xi)\Big)\omega_k(y)d\sigma(y).\nonumber\\
\end{eqnarray}
For $j=0$, we have by (\ref{tra})
\begin{eqnarray}\label{tau1}
|\varphi_{0}(x)|&=&\Big|\int_{\mathbb{S}^{d-1}}\tau_{x}(\mathcal{F}^{-1}_{k}(\psi_{0}))(-y) \omega_k(y)d\sigma(y)\Big|\nonumber
\\&\leq&\int_{\mathbb{S}^{d-1}}\Big(\int_{\mathbb{R}^d}\Big|\mathcal{F}^{-1}_{k}(\psi_{0})(\sqrt{\|x\|^2+\|y\|^2-2<y,\xi>})\Big|
d\mu_{x}(\xi)\Big)\omega_k(y)d\sigma(y).\nonumber\\
\end{eqnarray}
From (\ref{min}), (\ref{varj}), (\ref{tau}) and (\ref{tau1}), we get
for $j\geq 0$
\\$|\varphi_{j}(x)|$
\begin{eqnarray*}
&\leq&C2^{j(2\gamma+d)}\int_{\mathbb{S}^{d-1}}\Big(\int_{\mathbb{R}^d}
\frac{1}{(1+2^{j}\displaystyle\min_{g\in G}\|gx-y\|)^{2\gamma+d+1}}d\mu_{x}(\xi)\Big)\omega_k(y)d\sigma(y)
\\&\leq&C2^{j(2\gamma+d)}\int_{\mathbb{S}^{d-1}}\frac{1}{(1+2^{j}\displaystyle\min_{g\in G}\|gx-y\|)^{2\gamma+d+1}}\omega_k(y)d\sigma(y).
\end{eqnarray*}
If $\|x\|>2$, then $\|gx-y\|\geq \|x\|-1\geq\frac{\|x\|}{2}$ and we have
\begin{eqnarray}\label{var}
|\varphi_{j}(x)|&\leq& \frac{C2^{j(2\gamma+d)}}{(1+2^{j-1}\|x\|)^{2\gamma+d+1}}\nonumber\\&\leq& \frac{C2^{-j}}{\|x\|^{2\gamma+d+1}}\nonumber
\\&\leq& C \frac{2^{j}}{(1+\|x\|)^{2\gamma+d+1}}.
\end{eqnarray}
If $\|x\|\leq2$ then,
\begin{eqnarray}\label{var1}
|\varphi_{j}(x)|&\leq&\int_{\mathbb{S}^{d-1}}\frac{2^{j(2\gamma+d)}}{(1+2^{j}\displaystyle\min_{g\in G}\|gx-y\|)^{2\gamma+d+1}}\omega_k(y)d\sigma(y)\nonumber\\&\leq&
2^{j(2\gamma+d)}\int_{\displaystyle\{y\in\mathbb{S}^{d-1}; \min_{g\in G}\|gx-y\|\leq 2^{-j}\}}\omega_k(y)d\sigma(y)\nonumber\\&+&
2^{j(2\gamma+d)}\sum_{i=0}^{+\infty}2^{-(2\gamma+d+1)i}\int_{\displaystyle\{y\in\mathbb{S}^{d-1}; \min_{g\in G}\|gx-y\|\leq 2^{i+1-j}\}}\omega_k(y)d\sigma(y)
\nonumber\\&\leq&
C\Big(2^{j}+ 2^{j}\sum_{i=0}^{+\infty}2^{-2i}\Big)\leq
C\frac{2^{j}}{(1+\|x\|)^{2\gamma+d+1}}.
\end{eqnarray}
From (\ref{var}) and (\ref{var1}), we obtain \begin{eqnarray*}|\varphi_{j}(x)|\leq C \frac{2^{j}}{(1+\|x\|)^{2\gamma+d+1}}.\end{eqnarray*}
\end{proof}
\begin{lemma}
There exists a constant $C>0$ such that, for any $f\in\mathcal{S}(\mathbb{R}^d)$, $j\geq0$ and $\alpha>0$,
$$\int_{\widetilde{E}(\alpha)}d\nu_{k}(x)\leq\frac{ C2^{j}}{\alpha}\|f\|_{1,k}.$$
where $\widetilde{E}(\alpha)=\{x\in\mathbb{R}^d;
M_{\varphi_{j}}(f)(x)>\alpha\}$ and $C$ is a constant depending only
on $d$.
\end{lemma}
\begin{proof}
Let first prove that, for $f\in\mathcal{S}(\mathbb{R}^d)$ and
$x\in\mathbb{R}^d$,
\begin{eqnarray}\label{varp}\int_{\mathbb{R}^d}\tau_x(|\varphi_{j,r}|)(y)|f(y)|d\nu_{k}(y)\leq C 2^{j}\mathcal{M}_{k}(|f|)(x)\quad j\geq0,\end{eqnarray}
where $\mathcal{M}_{k}(f)$ is given by (\ref{max}).\\ We denote by
$\displaystyle A_i=\{y\in\mathbb{R}^d, \; r 2^i\leq
\|y\|<r2^{i+1}\}$, then\\
$\displaystyle\tau_x(|\varphi_{j,r}|)(y)=\tau_x(\sum_{i=-\infty}^{+\infty}|\varphi_{j,r}|.\chi_{A_i})(y)=\sum_{i=-\infty}^{+\infty}\tau_x(|\varphi_{j,r}|.\chi_{A_i})(y).$
\\From (\ref{lem}), one has
\begin{eqnarray*}
(|\varphi_{j,r}|.\chi_{A_i})(y)&\leq& C2^{j}\frac{r^{-(2\gamma+d)}}{ (1+\frac{\|y\|}{r})^{2\gamma+d+1}}\chi_{A_i}(y)
\\&\leq&C2^{j}\frac{r^{-(2\gamma+d)}}{
(1+2^i)^{2\gamma+d+1}}\chi_{A_i}(y),
\end{eqnarray*}
and since  $|\varphi_{j,r}|.\chi_{A_i}$ is a radial function, this
implies that
$$ \tau_x(|\varphi_{j,r}|.\chi_{A_i})(y)\leq
C2^{j}\frac{r^{-(2\gamma+d)}}{
(1+2^i)^{2\gamma+d+1}}\tau_x(\chi_{A_i})(y).$$
Using (\ref{3prmi}), we obtain\\
$\displaystyle\int_{\mathbb{R}^d}\tau_x(|\varphi_{j,r}|)(y)|f(y)|d\nu_{k}(y)$
\begin{eqnarray*}&\leq& C2^{j}\int_{\mathbb{R}^d}
\sum_{i=-\infty}^{+\infty}\frac{r^{-(2\gamma+d)}}{
(1+2^i)^{2\gamma+d+1}}\tau_x(\chi_{A_i})(y)|f(y)|d\nu_{k}(y)
\\&\leq& C2^{j}
\sum_{i=-\infty}^{+\infty}\frac{r^{-(2\gamma+d)}}{ (1+2^i)^{2\gamma+d+1}}\int_{\mathbb{R}^d}\chi_{A_i}(y)\tau_x(|f|)(y)d\nu_{k}(y)
\\&\leq& C2^{j}
\sum_{i=-\infty}^{+\infty}\frac{r^{-(2\gamma+d)}}{ (1+2^i)^{2\gamma+d+1}}\int_{B(0, r2^{i+1})}\tau_x(|f|)(y)d\nu_{k}(y)
\\&\leq& C2^{j}
\sum_{i=-\infty}^{+\infty}\frac{r^{-(2\gamma+d)}}{ (1+2^i)^{2\gamma+d+1}}(r2^{i+1})^{2\gamma+d}\mathcal{M}_{k}(|f|)(x)
\\&\leq& C2^{j}\mathcal{M}_{k}(|f|)(x).
\end{eqnarray*}
By the fact that,
\begin{eqnarray*}\left|\int_{\mathbb{R}^d}\tau_x(\varphi_{j,r})(y)f(y)d\nu_{k}(y)\right|
&\leq& \int_{\mathbb{R}^d}|\tau_x(\varphi_{j,r})(y)||f(y)|d\nu_{k}(y)\\&\leq& \int_{\mathbb{R}^d}\tau_x(|\varphi_{j,r}|)(y)|f(y)|d\nu_{k}(y),
\end{eqnarray*}
we deduce \\
 $$M_{\varphi_{j}}(f)(x)\leq C2^{j}\mathcal{M}_{k}(|f|)(x).$$
From (\ref{fai}), we conclude the proof of the lemma.
\end{proof}
\begin{lemma}
There exists a constant $C>0$ such that, for any $f\in\mathcal{S}(\mathbb{R}^d)$, $j\geq1$,
$$\|M_{\varphi_{j}}(f)\|_{2,k}\leq C2^{-j\frac{2\gamma+d-2}{2}}\|f\|_{2,k}.$$
\end{lemma}
\begin{proof}
(See \cite{S2}, \cite{Z})
For  $f\in \mathcal{S}(\mathbb{R}^d)$, we have from (\ref{ast})
 $$\mathcal{F}_{k}(f\ast_{k}\varphi_{j,r})(x)=\mathcal{F}_{k}(f)(x)\mathcal{F}_{k}(\varphi_{j,r})(x)=\mathcal{F}_{k}(f)(x)m_{j}(rx).$$ Put
 $$g_{j}(f)(x)=\Big(\int_{0}^{+\infty}|f\ast_{k}\varphi_{j,r}(x)|^{2}\frac{dr}{r}\Big)^{\frac{1}{2}},$$
  the Littlewood-Paley function associated to the function $\varphi_{j,r}$.\\
Using the Plancherel theorem and by (\ref{ast}), we obtain
\begin{eqnarray*}
\|g_{j}(f)\|_{2,k}^{2}&=&\int_{\mathbb{R}^d}|g_{j}(f)(x)|^{2}d\nu_{k}(x)\\&=&
\int_{\mathbb{R}^d}\Big(\int_{0}^{+\infty}|f\ast_{k}\varphi_{j,r}(x)|^{2}\frac{dr}{r}\Big)d\nu_{k}(x)\\&=&
\int_{0}^{+\infty}\Big(\int_{\mathbb{R}^d}|\mathcal{F}_{k}(f\ast_{k}\varphi_{j,r})(x)|^{2}d\nu_{k}(x)\Big)\frac{dr}{r}\\&=&
\int_{0}^{+\infty}\Big(\int_{\mathbb{R}^d}|\mathcal{F}_{k}(f)(x)|^{2}|m_{j}(rx)|^{2}d\nu_{k}(x)\Big)\frac{dr}{r}\\&\leq&
\|f\|_{2,k}^{2}\sup_{x\neq0}\int_{0}^{+\infty}|m_{j}(rx)|^{2}\frac{dr}{r}.
\end{eqnarray*}
Since the function $m_{j}$ is radial, the integral is independent of $x$.\\
By using the definition of $m_{j}$, we have that
\begin{eqnarray*}
\int_{0}^{+\infty}|m_{j}(rx)|^{2}\frac{dr}{r}&=&\int_{0}^{+\infty}|\mathcal{F}_{k}(\sigma)(rx)|^{2}|\psi_{j}(rx)|^{2}\frac{dr}{r}
\\&=&\int_{0}^{+\infty}|\mathcal{F}_{k}(\sigma)(rx)|^{2}|\psi_{1}(2^{-j}rx)|^{2}\frac{dr}{r}
\\&=&\int_{0}^{+\infty}|\mathcal{F}_{k}(\sigma)(2^{j}rx)|^{2}|\psi_{1}(rx)|^{2}\frac{dr}{r}.
\end{eqnarray*}
From (\ref{sig}), (\ref{sigm}) and (\ref{ps1}), we obtain
\begin{eqnarray*}
\int_{0}^{+\infty}|m_{j}(rx)|^{2}\frac{dr}{r}&\leq& C 2^{-j(2\gamma+d-1)}\int_{0}^{+\infty}\frac{|\psi_{1}(rx)|^{2}}{r^{2\gamma+d-1}}\frac{dr}{r}
\\&\leq& C 2^{-j(2\gamma+d-1)},
\end{eqnarray*}
this gives
\begin{eqnarray}\label{ggma}\|g_{j}(f)\|_{2,k}\leq C 2^{-\frac{j(2\gamma+d-1)}{2}}\|f\|_{2}.\end{eqnarray}
Put now $\tilde{\varphi}_{j,r}(x)=r\frac{d}{dr}\varphi_{j,r}(x)$ and
$\tilde{g_{j}}(f)$ the Littlewood-Paley function associated to
$\tilde{\varphi}_{j,r}$, then
 $$\tilde{g_{j}}(f)(x)=\Big(\int_{0}^{+\infty}|f\ast_{k}\tilde{\varphi}_{j,r}(x)|^{2}\frac{dr}{r}\Big)^{\frac{1}{2}}=
 \Big(\int_{0}^{+\infty}r|\frac{d}{dr}\varphi_{j,r}\ast_{k} f(x)|^{2}dr\Big)^{\frac{1}{2}}.$$
Similarly, with the use of  (\ref{sig}) and (\ref{par}), we have
\begin{eqnarray}\label{gtil}
\|\tilde{g_{j}}(f)\|_{2,k}&=&
\int_{\mathbb{R}^d}\Big(\int_{0}^{+\infty}r|\frac{d}{dr}\varphi_{j,r}\ast_{k}
f(x)|^{2}dr\Big)d\nu_{k}(x)\nonumber\\&=& C
2^{-j\frac{(2\gamma+d-1)}{2}}\|f\|_{2,k}.
\end{eqnarray}
Since $\displaystyle\lim_{r\rightarrow+\infty}f\ast_{k} \varphi_{j,r}(x)=0$, we have
\begin{eqnarray*}
|f\ast_{k}
\varphi_{j,r}(x)|^{2}&=&-2Re\Big(\int_{r}^{+\infty}\overline{\varphi_{j,s}\ast_{k}
f(x)}\frac{d}{ds}\varphi_{j,s}\ast_{k} f(x) ds\Big)
\\&=&-2Re\Big(\int_{r}^{+\infty}\overline{\varphi_{j,s}\ast_{k} f(x)}\tilde{\varphi}_{j,s}\ast_{k} f(x)\frac{ds}{s}\Big)
\\&\leq&2\int_{r}^{+\infty}|\overline{\varphi_{j,s}\ast_{k} f(x)}||\tilde{\varphi}_{j,s}\ast_{k} f(x)|\frac{ds}{s}.
\end{eqnarray*}
 Using Cauchy-Schwartz's inequality, we deduce that
$$\displaystyle\sup_{r>0}|f\ast_{k} \varphi_{j,r}(x)|^{2}\leq 2g_{j}(f)(x)\tilde{g_{j}}(f)(x).$$
Integrating over $\mathbb{R}^d$ and  using again the Cauchy-Schwartz
inequality, we obtain from (\ref{ggma}) and (\ref{gtil})
$$\|M_{\varphi_{j}}(f)\|_{2,k}\leq C2^{-j\frac{2\gamma+d-2}{2}}\|f\|_{2,k}.$$
\end{proof}
\begin{lemma}
There exists a constant $C>0$ such that, for $f\in\mathcal{S}(\mathbb{R}^d)$
\begin{eqnarray*}\|M_{\varphi_{j}}(f)\|_{\infty,k}\leq C2^{j}\|f\|_{\infty,k}\quad  j\geq0.\end{eqnarray*}
\end{lemma}
\begin{proof}
From (\ref{3frad}) and (\ref{lem}), one has for $f\in\mathcal{S}(\mathbb{R}^d)$ and $x\in\mathbb{R}^d$,
\begin{eqnarray*}
|f\ast_{k} \varphi_{j,r}(x)|&\leq&\int_{\mathbb{R}^d}|f(y)| |\tau_{x}\varphi_{j,r}(y)|d\nu_{k}(y)
\\&\leq&\|f\|_{\infty,k}\int_{\mathbb{R}^d} |\varphi_{j,r}(y)|d\nu_{k}(y)
\\&\leq&C 2^{j}\|f\|_{\infty,k}\int_{\mathbb{R}^d} \frac{1}{(1+\|x\|)^{2\gamma+d+1}}d\nu_{k}(y)
\\&\leq&C 2^{j}\|f\|_{\infty,k},
\end{eqnarray*}
which gives the result.
\end{proof}
\begin{remark}We observe that for $j=0$ and from Lemmas 2.2, 2.4, we obtain by
interpolation (see \cite{S3}) that $M_{\varphi_{j}}(f)$ is of
strong-type (p,p) with $1<p\leq +\infty$.\end{remark}
\noindent\textbf{Proof of theorem 2.1.} According to Remark 2.1,
we deduce by interpolation from Lemmas 2.2, 2.3 that
$$\|M_{\varphi_{j}}(f)\|_{p,k}\leq
C2^{-j(2\gamma+d-\frac{2\gamma+d}{p}-1)}\|f\|_{p,k},$$ for
$1<p\leq2$ and $j\geq0$.\\ Similarly, according to Remark 2.1, we
get by interpolation from Lemmas 2.3, 2.4 that
$$\|M_{\varphi_{j}}(f)\|_{p,k}\leq
C2^{-j(\frac{2\gamma+d}{p}-1)}\|f\|_{p,k},$$ for $2\leq
p\leq +\infty$ and $j\geq0$.\\
Since for $\frac{2\gamma+d}{2\gamma+d-1}<p<2\gamma+d$, we have
$$\sum_{j\geq0}^{+\infty}2^{-j(2\gamma+d-\frac{2\gamma+d}{p}-1)}<+\infty\quad\mbox{and}\quad\sum_{j\geq0}^{+\infty}2^{-j(\frac{2\gamma+d}{p}-1)}<+\infty.$$
This yields using (\ref{princ})
$$\|M(f)\|_{p,k}\leq C\,\|f\|_{p,k},$$ which completes
the proof.
\begin{remark}
The case $2\gamma+d=2$ implies that $k\equiv0$  and was proved by
Bourgain in \cite{B}.
\end{remark}


\addcontentsline{toc}{section}{Bibliographie}
\begin{thebibliography}{}
\bibitem{B}
\textbf{J. Bourgain,}
\emph{Averages in the plane over convex curves and maximal operators,
J. Analyse Math. 47 (1986), 69-85.}

\bibitem{DA}
\textbf{F. Dai and H. Wang,}
\emph{A transference theorem for the Dunkl transform and its applications,
Journal of Functional Analysis 258 (2010), no. 12, 4052--4074.}

\bibitem{D}
\textbf{C. F. Dunkl,}
\emph{Differential--Difference operators associated to reflextion groups,
Trans. Amer. Math. 311 (1989), no. 1, 167--183.}

\bibitem{J}
\textbf{M.F.E. de Jeu,}
\emph{The Dunkl transform,
Invent. Math. 113 (1993), no. 1, 147-162.}

\bibitem{mejj}
\textbf{H. Mejjaoli and K. Trimèche,}
\emph{On a mean value property associated with the Dunkl Laplacian
operator and applications, Integral Transform. Spec. Funct. 12 (2001), 279-302.}

\bibitem{R1}
\textbf{M. R\"osler,}
\emph{Generalized Hermite polynomials
and the heat equation for Dunkl operators,
Comm. Math. Phys. 192 (1998), no. 3, 519--542.}

\bibitem{R2}
\textbf{M. R\"osler,}
\emph{Positivity of Dunkl's intertwining operator,
Duke Math. J. 98 (1999), 445--463.}

\bibitem{R3}
\textbf{M. R\"osler,}
\emph{Dunkl operators\,: theory and applications,
Orthogonal polynomials and special functions
(Leuven, 2002),
Lect. Notes Math. 1817, Springer--Verlag (2003), 93--135.}

\bibitem{R4}
\textbf{M. R\"osler,}
\emph{Bessel--type signed hypergroup on $\mathbb{R}$
, Probability measures on groups and related structures,
XI (Oberwolfach, 1994), World Sci. Publ., River Edge, NJ, 1995, pp. 292–304.}

\bibitem{R5}
\textbf{M. R\"osler,}
\emph{A positive radial product formula for the Dunkl kernel,
Trans. Amer. Math. Soc. 355 (2003), no. 6, 2413--2438.}

\bibitem{R6}
\textbf{M. R\"osler,}
\emph{Markov Processes Related With Dunkl Operators,
Advances in Applied Mathematics 21, 575--643 (1998).}

\bibitem{S1}
\textbf{E.M. Stein,}
\emph{Maximal functions: spherical mean,
Proc. Nat. Acad. Sci. U.S.A. 73(1976), 2174-2175.}

\bibitem{S2}
\textbf{E.M. Stein,}
\emph{Singular Integrals and Differentiability Properties of Functions,
Princeton Mathematical Series 30. Princeton University Press,
Princeton, NJ, 1970.}

\bibitem{S3}
\textbf{E.M. Stein and G. Weiss,}
 \emph{Introduction to Fourier Analysis on
Euclidean Spaces, Princeton Mathematical Series 32. Princeton
University Press, Princeton, NJ, 1971.}

\bibitem{TX}
\textbf{S. Thangavelu and Y. Xu,}
\emph{Convolution operator and maximal function for the Dunkl transform,
J. Anal. Math. 97 (2005), 25--56.}

\bibitem{Z}
\textbf{A. Zygmund,}
\emph{Trigonometric Series, Vol. I, II. Third edition. Cambridge
Mathematical Library. Cambridge University Press, Cambridge,
2002.}
\end{thebibliography}
\end{document}